\documentclass{amsart}
\usepackage{amssymb,enumerate}
\usepackage[chatter]{rotating}
\usepackage{epsfig}
\usepackage{fancyvrb}
\newcommand{\RomanNumeralCaps}[1]
    {\MakeUppercase{\romannumeral #1}}
\newtheorem{lemma}{Lemma}

\newtheorem{theorem}[lemma]{Theorem}
\newtheorem{cor}[lemma]{Corollary}
\newcommand{\ad}{{\rm ad}}
\newcommand{\F}{{\mathbb F}}
\newcommand{\Z}{{\mathbb Z}}

\newcommand{\lm}{\lambda}

\newcommand{\al}{\alpha}
\newcommand{\bt}{\beta}
\newcommand{\sg}{\sigma}

\newcommand{\h}{\hat{\mathcal H}}
\newcommand{\ii}{\bar{\imath}}

\title[$2$-generated axial algebras]{A note on $2$-generated symmetric axial algebras of Monster type}
\author{Clara Franchi and Mario Mainardis}

\address{Dipartimento di Matematica e Fisica,
Universit\`a Cattolica del Sacro Cuore,
Via Musei 41,
I-25121 Brescia, Italy}
\email{clara.franchi@unicatt.it}

\address{Dipartimento di Scienze Matematiche, Informatiche e Fisiche, 
Universit\`a degli Studi di Udine, via delle Scienze 206,
I-33100 Udine, Italy}
\email{mario.mainardis@uniud.it}

\begin{document}
\begin{abstract}
In~\cite{Yabe}, Yabe gives an almost complete classification of primitive symmetric $2$-generated axial algebras of Monster type. In this note, we construct a new infinite-dimensional primitive $2$-generated symmetric axial algebra of Monster type $(2, \frac{1}{2})$ over a field of characteristic $5$, and use this algebra to complete the last case left open in Yabe's classification. 
\end{abstract}
\maketitle


\section{Introduction}
Axial algebras of Monster type have been introduced in~\cite{HRS} by Hall, Rehren and Shpectorov, in order to generalise some subalgebras of the Griess algebra (defined Majorana algebras by Ivanov~\cite{Ivanov}) and create a new tool for better understanding, and possibly unifying, the classification of finite simple groups. They have recently appeared also in other branches of mathematics (see~\cite{tkb, tk1}). In~\cite{RT,R} Rehren started a systematic study of primitive $2$-generated axial algebras of Monster type, constructing several classes of new algebras. More examples have been found by Galt et al.~\cite{DM},  and, independently, by Yabe~\cite{Yabe}.  
In~\cite{Yabe}, Takahiro Yabe obtained an almost complete classification of the primitive $2$-generated symmetric axial algebras of Monster type. Yabe left open only the case of algebras over a field of characteristic $5$ and axial dimension greater than $5$. Indeed something surprising happens in the latter case, namely we show that, over a field $\F$ of  characteristic $5$, there exists a new infinite-dimensional primitive $2$-generated symmetric axial algebra $\h$ of Monster type $(2, \frac{1}{2})$ such that any primitive $2$-generated symmetric axial algebras of Monster type $(2, \frac{1}{2})$  is isomorphic to a quotient of $\h$ (in particular, the Highwater algebra $\mathcal H$~\cite{FMS1} is a proper factor of $\h$ over an infinite-dimensional ideal). As a corollary we complete Yabe's classification. 
 
 For the definitions and further motivation refer to~\cite{HRS, RT, R}, for the notation and the basic properties of axial algebras refer to~\cite{FMS1}. In particular, throughout this paper, $\F$ is a field of characteristic $5$. For a $2$-generated symmetric axial algebra $V$ of Monster type $(\al, \bt)$ over  $\F$, let $a_0, a_1$ be the generating axes of $V$.  For $i\in \{0,1\}$,  let $\tau_i$ be the Miyamoto involution associated to $a_i$. Set $\rho:=\tau_0\tau_1$, and, for $i\in \Z$, $a_{2i}:=a_0^{\rho^i}$ and $a_{2i+1}:=a_1^{\rho^i}$. Note that, since $\rho$ is an automorphism of $V$, for every $j\in \Z$,  $a_j$ is an axis. Denote by $\tau_j:=\tau_{a_j}$ the corresponding Miyamoto involution. The algebra $V$ is symmetric if it admits an algebra automorphism $f$ that swaps $a_0$ and $a_1$, whence, for every $i\in \Z$, $a_i^f=a_{-i+1}$. Let $\sigma_{i}$ be the element of $\langle \tau_0, f\rangle $ that swaps $a_0$ with $a_{i}$.
Since, by~\cite[Lemma~4.2]{FMS1}, for $ n\in \Z_+ $ and $i,j\in \Z$ such that $i\equiv_n j$, 
$$
a_ia_{i+n}-\bt(a_i+a_{i+n})=a_ja_{j+n}-\bt(a_j+a_{j+n}),
$$
we can define 
\begin{equation}\label{sn}
s_{\bar  \imath,n}:=a_ia_{i+n}-\beta (a_i+a_{i+n}). 
\end{equation}
where $\bar \imath$ denotes the congruence class $i+n\Z$.
Since $V$ is primitive, there is a linear function $\lambda_{a_0}:V\to \F$  such that every $v$ can be written in a unique way as $v=\lambda_{a_0}(v)a_0+v_0+v_\al+v_\bt$, where $v_0$, $v_\al$, $w_\bt$ are $0$-, $\al$-, $\bt$-eigenvectors for $\ad_{a_0}$, respectively. For $i\in \Z$, set $\lm_i:=\lambda_{a_0}(a_i)$ and let
 \begin{equation}\label{ai}
 a_i=\lambda_{a_0}(a_i)  a_0+u_i+v_i +w_i
 \end{equation}
 be the decomposition of $a_i$ into $ad_{a_0}$-eigenvectors, where $u_i$ is a $0$-eigenvector, $v_i$ is an $\al$-eigenvector and $w_i$ is a $\bt$-eigenvector.
 
 Following Yabe~\cite{Yabe}, we denote by $D$ the positive integer such that $\{a_0, \ldots , a_{D-1}\}$ is a basis for the linear span $\langle a_i \mid i\in \Z\rangle$ of the set of the axes $a_i$'s.  $D$ is called the {\it axial dimension} of $V$. Our results are the following
 
 \begin{theorem}\label{thm1}
Let $V$ be a primitive $2$-generated symmetric axial algebra of Monster type $(2, \frac{1}{2})$ over a field of characteristic $5$. If $\lm_1=\lm_2=1$, then $V$ is isomorphic to a quotient of the algebra $\h$.
\end{theorem}

\begin{cor}\label{thms}
Let $V$ be a primitive $2$-generated symmetric axial algebra of Monster type $(\al, \bt)$ over a field of characteristic $5$. If $D\geq 6$, then $V$ is isomorphic to a quotient of one of the following:
\begin{enumerate}
\item the algebra $6A_{\alpha}$, as defined in~\cite{R};
\item the algebra $V_8(\al)$, as defined in~\cite{DM};
\item the algebra $\h$, as defined in Section~\ref{H}.
\end{enumerate}
\end{cor} 
Note that, for every $\al$, Rehren's algebra $6A_\al$ and Yabe's algebra $\mbox{\RomanNumeralCaps{6}}_2(\al, \frac{-\al^2}{4(2\al-1)})$ coincide and the $8$-dimensional algebra $V_8(\al)$ of type $(\al, \frac{\al}{2})$ constructed in~\cite{DM} coincides with Yabe's algebra $\mbox{\RomanNumeralCaps{6}}_1(\al, \frac{\al}{2})$.
Furthermore, remarkably, over a field of characteristic $5$, the Highwater algebra $\mathcal H$ (see~\cite{HW} and~\cite{Yabe}) is isomorphic to a quotient of $\h$, Yabe's algebras $\mbox{\RomanNumeralCaps{5}}_1(2,\frac{1}{2})$ and $\mbox{\RomanNumeralCaps{5}}_2(2,\frac{1}{2})$, and Rehren's algebra $5A_2$ are all isomorphic, and are in turn a quotient of $\mathcal H$. Finally, also the algebra $6A_{2}$ is a quotient algebra of $\h$.

\section{The algebra $\h$}\label{H}
In this section, for every $i\in \Z$, denote by $\ii$ the congruence class $i+3\Z$. Let ${\h}$ be an infinite-dimensional 
$\F$-vector space with basis $\mathcal B:=\{\hat a_i,\hat s_{\bar  0,j}, \hat s_{\bar  1, 3k}, \hat s_{\bar  2,3k} \mid i\in \Z, \: j, k \in \Z_+ \}$,
$$
{\h}:=\bigoplus_{ i\in \Z} \F \hat a_i\oplus\bigoplus_{ j\in \Z_+} \F \hat s_{\bar  0,j}\oplus\bigoplus_{ k\in \Z_+}(\F \hat s_{\bar  1,3k}\oplus\F \hat s_{\bar  2,3k}).
$$
Set $\hat s_{\bar  0,0}:=0$ and, if $j \not \equiv_3 0$, $\hat s_{\bar  1,j}:=\hat s_{\bar  0,j}=:\hat s_{\bar  2,j}$. 
Let $\hat \tau_0$ and $\hat f$ be the linear maps of $\h$ defined on the basis elements by
$$
\hat a_i^{\hat \tau_0}=\hat a_{-i}, \:(\hat s_{\bar  0,j})^{\hat \tau_0}=\hat s_{\bar  0,j}, \:(\hat s_{\bar  1,3k})^{\hat \tau_0}=\hat s_{\bar  2,3k},\: \mbox{ and }\: (\hat s_{\bar  2,3k})^{\hat \tau_0}=\hat s_{\bar  1,3k},
$$
$$
\hat a_i^{\hat f}=\hat a_{-i+1}, \:(\hat s_{\bar  0,j})^{\hat f}=\hat s_{\bar  0,j} \:\mbox{ if } \: j\not \equiv_3 0,
$$
and
$$
 (\hat s_{\bar  0,3k})^{\hat f}=\hat s_{\bar  1,3k}, \:(\hat s_{\bar  1,3k})^{\hat f}=\hat s_{\bar  0,3k}, \:(\hat s_{\bar  2,3k})^{\hat f}=\hat s_{\bar  2,3k}.
$$
Define a commutative non-associative product on ${\h}$ extending by linearity the 
following values on  the basis elements (where $\delta_{\ii \bar r}$ denotes the Kronecker delta and $(-1)\ast {\bar 1}:=-1$, $(-1)\ast {\bar 2}:=1$, and $0\ast \bar t:=0$ for every $t\in \Z$):
\medskip
\begin{enumerate}
\item [($\h_1$)] $\hat a_i\hat a_j:=-2(\hat a_i+\hat a_{j})+\hat s_{\bar \ii,|i-j|}$, \\
%
 %
\item  [($\h_2$)] $\hat a_i \hat s_{\bar  r,j}:= -2\hat a_i+(\hat a_{i-j}+\hat a_{i+j})- \hat s_{\bar  r,j}-(\delta_{\ii \bar r} -1)\ast {(\ii-\bar r)}(\hat s_{\bar  r-\bar 1,j}-\hat s_{\bar  r+\bar 1,j})$,\\  
\item  [($\h_3$)] $\hat s_{\bar  r,i} \hat s_{\bar  t,j}:= 
2(\hat s_{\bar  r,i}+\hat s_{\bar  t,j})-2(\hat s_{\bar  0,{|i-j|}}+\hat s_{\bar  1, {|i-j|}}+\hat s_{\bar  2, {|i-j|}}
+\hat s_{\bar  0,{i+j}}+\hat s_{\bar  1,{i+j}}+\hat s_{\bar  2,{i+j}})$,  if $ \{i,j\}\not \subseteq 3\Z$,\\
\item  [($\h_4$)]  $\hat s_{\bar  0,3h} \hat s_{\bar  0,3k}:=2(\hat s_{\bar  0,3h}+\hat s_{\bar  0,3k})-(\hat s_{\bar  0,3|h-k|}+\hat s_{\bar  0,3(h+k)})$, \\
\item  [($\h_5$)] $\hat s_{\bar  0,3h} \hat s_{\bar  1,3k}:=2(\hat s_{\bar  0,3h}+\hat s_{\bar  1,3h}-\hat s_{\bar  2,3h}+\hat s_{\bar  0,3k}+\hat s_{\bar  1,3k}-\hat s_{\bar  2,3k})-(\hat s_{\bar  0,3|h-k|}+\hat s_{\bar  1,3|h-k|}-\hat s_{\bar  2,3|h-k|}+\hat s_{\bar  0,3(h+k)}+\hat s_{\bar  1,3(h+k)}-\hat s_{\bar  2,3(h+k)})$,\\
\item  [($\h_6$)] $\hat s_{\bar  0,3h} \hat s_{\bar  2,3k}:=2(\hat s_{\bar  0,3h}-\hat s_{\bar  1,3h}+\hat s_{\bar  2,3h}+\hat s_{\bar  0,3k}-\hat s_{\bar  1,3k}+\hat s_{\bar  2,3k})-(\hat s_{\bar  0,3|h-k|}-\hat s_{\bar  1,3|h-k|}+\hat s_{\bar  2,3|h-k|}+\hat s_{\bar  0,3(h+k)}-\hat s_{\bar  1,3(h+k)}+\hat s_{\bar  2,3(h+k)})$,\\
\item  [($\h_7$)]  $\hat s_{\bar  1,3h} \hat s_{\bar  2,3k}:=2(-\hat s_{\bar  0,3h}+\hat s_{\bar  1,3h}+\hat s_{\bar  2,3h}-\hat s_{\bar  0,3k}+\hat s_{\bar  1,3k}+\hat s_{\bar  2,3k})-(-\hat s_{\bar  0,3|h-k|}+\hat s_{\bar  1,3|h-k|}+\hat s_{\bar  2,3|h-k|}-\hat s_{\bar  0,3(h+k)}+\hat s_{\bar  1,3(h+k)}+\hat s_{\bar  2,3(h+k)})$.
\end{enumerate} 
\medskip

We now introduce some eigenvectors for $\ad_{\hat a_0}$ and study how they multiply.
For $i\in \Z_+$, set 
$$
\hat u_i:=-2\hat a_0+(\hat a_i+\hat a_{-i})+2\hat s_{\bar  0,i},
$$ 
$$
\hat v_i:=-2\hat a_0+(\hat a_i+\hat a_{-i})-\hat s_{\bar  0,i},
$$
$$
\hat  w_i:=\hat a_i-\hat a_{-i},
$$
$$
\overline u_{i}:=-2\hat  a_0+(\hat a_{-i}+\hat a_{i}) -(\hat s_{\bar  0,i}+\hat s_{\bar  1,i}+\hat s_{\bar  2,i}), 
$$
$$
  \overline w_{i}:=\hat s_{\bar  1,i}-\hat s_{\bar  2,i}.
$$
Then, the $\hat u_i$'s and $\overline u_{i}$'s are $0$-eigenvectors for $\ad_{\hat a_0}$, the $\hat v_i$'s are $2$-eigenvectors for $\ad_{\hat a_0}$, the  $\hat w_i$'s and $\overline w_{i}$'s are $-2$-eigenvectors for $\ad_{\hat a_0}$. Note that, if $i\not \equiv_3 0$, $\overline u_{i}=\hat u_i$ and $\overline w_{i}=0$. 
Moreover, it will be convenient to use the following notation: 
for $i,j\in\Z_+$,  set
$$\hat c_j:=-2\hat a+(\hat a_{-j}+\hat a_j),
$$
$$
\hat c_{i,j}:=-2\hat c_i-2\hat c_j+\hat c_{|i-j|}+\hat c_{i+j}, 
$$  
\begin{equation*}
\hat \sg_{i,j}:= \hat s_{\bar 0,i}+\hat s_{\bar 0,j}+\hat s_{-\bar \imath ,|i-j|}+\hat s_{\bar \imath, |i-j|}
+\hat s_{-\bar \imath ,i+j}+\hat s_{\bar \imath, i+j},
\end{equation*} 
$$
\hat u_{i,j}:=-2\hat u_i-2\hat u_j+\hat u_{|i-j|}+\hat u_{i+j},
$$ 
$$\hat v_{i,j}:=
-2\hat v_i-2\hat v_j+\hat v_{|i-j|}+\hat v_{i+j}.
$$ 
We collect in the following lemma the main relations among the above vectors.
\begin{lemma} \label{transition}
For all $i,j\in\Z_+$, we have 
\begin{enumerate}
\item $\hat u_j=\hat c_j+2\hat s_{\bar 0,j}$, 
\item $\hat v_j=\hat c_j-\hat s_{\bar 0,j}$, 
\item $\overline{u}_j=\hat c_j-(\hat s_{\bar 0,j}+\hat s_{\bar 1,j}+\hat s_{\bar 2,j})$,
\item $\hat u_{i,j}=\hat c_{i,j}+ \hat s_{\bar 0,i}+\hat s_{\bar 0,j}+2\hat s_{\bar 0, |i-j|}+2\hat s_{\bar 0, i+j}$,
\item $\hat v_{i,j}=\hat c_{i,j}+2( \hat s_{\bar 0,i}+\hat s_{\bar 0,j})-(\hat s_{\bar 0, |i-j|}+\hat s_{\bar 0, i+j})$.
\item $\hat c_i\hat c_j=\hat \sg_{i,j}$,  
\item $\hat c_i\hat s_{\bar r,j}=\hat c_{i,j}=\hat c_j\hat s_{\bar r,i}$.
\end{enumerate}
\end{lemma}

\begin{proof}
The first five assertions are immediate. A straightforward computation gives the sixth:
\begin{eqnarray*}
\hat c_i\hat c_j&=&(-2\hat a_0+(\hat a_{-i}+\hat a_i))(-2\hat a_0+(\hat a_{-j}+\hat a_j))\\
&=&-\hat a_0-2(-2\hat a_{-i}+\hat a_0+\hat s_{\bar 0,i}-2\hat a_i-2\hat a_0+\hat s_{\bar 0,i})\\
&&\:\:\:\:\:\:\:\:-2(-2\hat a_0-2\hat a_{-j}+\hat s_{\bar 0,j}-2\hat a_0-2\hat a_j+\hat s_{\bar 0,j})\\
&&\:\:\:\:\:\:\:\:+(-2\hat a_{-i}-2\hat a_{-j}+\hat s_{-\bar \imath ,|i-j|}
-2\hat a_i-2\hat a_{-j}+\hat s_{\bar \imath, i+j}\\
&&\:\:\:\:\:\:\:\:\:\:\:\:\:\:\:\:-2\hat a_{-i}-2\hat a_j+\hat s_{-\bar \imath ,i+j}-2\hat a_i-2\hat a_j+\hat s_{\bar \imath, |i-j|})\\
&=&\hat s_{\bar 0,i}+\hat s_{\bar 0,j}+\hat s_{-\bar \imath ,|i-j|}+\hat s_{\bar \imath, |i-j|}
+\hat s_{-\bar \imath ,i+j}+\hat s_{\bar \imath, i+j}=\hat \sg_{i,j} .
\end{eqnarray*}

Similarly, for the seventh, we have
$$
\begin{array}{l}
\hat c_i\hat s_{\bar r,j}=-(2\hat a_0-(\hat a_{-i}+\hat a_i))\hat s_{\bar r,j}=\\
=-2[-2\hat a_0+(\hat a_{-j}+\hat a_j)-\hat s_{\bar r,j}-(\delta_{\bar 0\bar r} -1)\ast {(\bar 0-\bar r)}(\hat s_{\bar r-\bar 1,j}-\hat s_{\bar r+\bar 1,j})]\\
+(-2\hat a_{-i}+(\hat a_{-i-j}+\hat a_{-i+j})-\hat s_{\bar r,j}-( \delta_{-\bar \imath \bar r}-1)\ast {(-\ii-\bar r)}(\hat s_{\bar r-\bar 1,j}-\hat s_{\bar r +\bar 1,j})\\
+(-2\hat a_{i}+(\hat a_{i-j}+\hat a_{i+j})-\hat s_{\bar r,j}-(\delta_{\bar \imath \bar r}-1)\ast {(\ii-\bar r)}(\hat s_{\bar r-\bar 1,j}-\hat s_{\bar r +\bar 1,j}).\\
\end{array}
$$
 Here, $\hat s_{\bar r,j}$, $\hat s_{\bar r-\bar 1,j}$, and $\hat s_{\bar r+\bar 1,j}$ cancel and the first equality of the last claim follows after the terms are rearranged. Since, by the definition, $\hat c_{i,j}$ is symmetric in $i$ and $j$, we have $\hat c_{i,j}=
\hat c_{j,i}=\hat c_j\hat s_{\bar r, i}$. 
\end{proof}

Note that also $\hat \sg_{i,j}$ is symmetric in $i$ and $j$. This evident by Lemma~\ref{transition}.(4), since $\h$ is commutative. Two more relations will be useful in the sequel.

\begin{lemma}\label{usefulformula}
For all $i,j\in\Z_+$, we have
\begin{enumerate}
\item 
$
\hat s_{\bar 0,i}(\hat s_{\bar 0,j}+\hat s_{\bar 1,j}+\hat s_{\bar 2,j})=\hat s_{\bar 0,i}+\hat s_{\bar 0,j}+2(\hat s_{\bar 0,|i-j|}+\hat s_{\bar 0,i+j}),
$
\item $\hat u_i-\overline{u}_{i}=-2\hat s_{\bar 0,i}+\hat s_{\bar 1,i}+\hat s_{\bar 2,i}$. 
\end{enumerate}
\end{lemma}
\begin{proof}
A straightforward computation gives the claims.
\end{proof}
Now are now ready to compute the products of the vectors $\hat u_j$, $\overline{u}_j$,  and $\hat v_j$.

\begin{lemma} \label{products u v}
For all $i,j\in\Z_+$, we have
\begin{enumerate}
\item  $\hat u_i\hat u_j=-\hat u_{i,j}-2(\hat u_{|i-j|}-\overline{u}_{|i-j|})-2(\hat u_{i+j}-\overline{u}_{i+j})$, 
\item $u_iv_j=v_{i,j}$, 
\item $v_iv_j=-2u_{i,j}-(\hat u_{|i-j|}-\overline{u}_{|i-j|})-(\hat u_{i+j}-\overline{u}_{i+j})$,
\item $\hat u_i\overline u_j=-\hat u_{i,j}$,
\item $\hat v_i\overline u_j=\hat v_{i,j}$.
\end{enumerate}
\end{lemma}
\begin{proof}
By Lemma \ref{transition},
\begin{eqnarray*}
\hat u_i\hat u_j&=&(\hat c_i+2\hat s_{\bar 0,i})(\hat c_j+2\hat s_{\bar 0,j})\\
&=&\hat c_i\hat c_j+2\hat c_i\hat s_{\bar 0,j}+2\hat s_{\bar 0,i}\hat c_j-\hat s_{\bar 0,i}\hat s_{\bar 0,j}\\
&=&\hat \sigma_{i,j}+2\hat c_{i,j}+2\hat c_{i,j}-\hat s_{\bar 0,i}\hat s_{\bar 0,j}\\
&=&\hat \sigma_{i,j}-\hat c_{i,j}-\hat s_{\bar 0,i}\hat s_{\bar 0,j}.
\end{eqnarray*}
Assume $i\equiv_3 0$. Then $\hat \sigma_{i,j}=\hat s_{\bar 0, i}+\hat s_{\bar 0, j}+2(\hat s_{\bar 0, |i-j|}+\hat s_{\bar 0,i+ j})$. Moreover, 
\begin{equation}\label{prod}
\hat s_{\bar 0,i}\hat s_{\bar 0,j}= 2(\hat s_{\bar 0, i}+\hat s_{\bar 0, j})-(\hat s_{\bar 0, |i-j|}+\hat s_{\bar 0, i+j}).
\end{equation}
This is immediate if $j\equiv_3 0$, since in this case $(\h_4)$ holds. If $j\not \equiv_3 0$, then $|i-j|\not \equiv_3 0$ and $i+j\not \equiv_3 0$. Thus $\hat s_{\bar 0, |i-j|}=\hat s_{\bar 1, |i-j|}=\hat s_{\bar 2, |i-j|}$, $\hat s_{\bar 0, i+j}=\hat s_{\bar 1, i+j}=\hat s_{\bar 0, i+j}$ and $(\h_3)$ reduces to (\ref{prod}). Hence, by Lemma~\ref{transition}.(4)
\begin{eqnarray*}
\hat u_i\hat u_j&=&-\hat c_{i,j}+\hat s_{\bar 0, i}+\hat s_{\bar 0, j}+2(\hat s_{\bar 0, |i-j|}+2\hat s_{\bar 0,i+ j}) -2(\hat s_{\bar 0, i}+\hat s_{\bar 0, j})\\
&&+(\hat s_{\bar 0, |i-j|}+\hat s_{\bar 0, i+j})\\
&=&-\hat c_{i,j}- (\hat s_{\bar 0,i}+\hat s_{\bar 0,j})-2(\hat s_{\bar 0, |i-j|}+\hat s_{\bar 0, i+j})=-\hat u_{i,j}.\\
\end{eqnarray*}
Assume $i\not \equiv_3 0$ and $j\not \equiv_3 0$. Then $\hat \sigma_{i,j}=\hat s_{\bar 0, i}+\hat s_{\bar 0, j}+\hat s_{\bar 1, |i-j|}+\hat s_{\bar 2,|i-j|}+\hat s_{\bar 1,i+ j}+\hat s_{\bar 2,i+ j}$, while the product $\hat s_{\bar 0,i}\hat s_{\bar 0,j}$ is given by ($\h_2$). Hence
\begin{eqnarray*}
\hat u_i\hat u_j&=&-\hat c_{i,j}+\hat s_{\bar 0, i}+\hat s_{\bar 0, j}+\hat s_{\bar 1, |i-j|}+\hat s_{\bar 2,|i-j|}+\hat s_{\bar 1,i+ j}+\hat s_{\bar 2,i+ j}-2(\hat s_{\bar 0, i}+\hat s_{\bar 0, j})\\
&& +2(\hat s_{\bar 0, |i-j|}+\hat s_{\bar 1, |i-j|}+\hat s_{\bar 2, |i-j|}+\hat s_{\bar 0, i+j}+\hat s_{\bar 1, i+j}+\hat s_{\bar 2, i+j})\\
&=&-\hat u_{i,j}-2(\hat u_{|i-j|}-\overline u_{|i-j|})-2(\hat u_{i+j}-\overline u_{i+j}).
\end{eqnarray*}
This proves the first assertion. The remaining assertions follow with similar computations, using Lemma~\ref{usefulformula}.
\end{proof}

\begin{theorem}\label{HW5}
The algebra ${\h}$ defined above is a primitive $2$-generated symmetric axial algebra of Monster type $(2, \frac{1}{2})$ over any field of characteristic $5$.  
\end{theorem}
\begin{proof}
Remind that in characteristic $5$, $\frac{1}{2}=-2$. It is easy to see that the maps $\hat \tau_0$ and $\hat f$ are algebra automorphisms of $\h$ and that the map $\hat \theta:=\tau_0\hat f$ induces on the set $\{\hat a_i\mid i\in \Z\}$ the translation $\hat a_{i}\mapsto \hat a_{i+1}$. Let $H:=\langle\!\langle 
a_0,a_1\rangle\!\rangle$ be the subalgebra of $\h$ generated by $\hat a_0$ and $\hat a_1$. Note that $\hat s_{\bar  0,1}=\hat a_0\hat a_1+2(\hat a_0+\hat a_1)\in H$. Also, 
$\hat a_{-1}=\hat a_0\hat s_{\bar  0,1}+2\hat a_0-\hat a_1+\hat s_{\bar  0,1}\in H$. This gives us $\hat a_{-1}\in H$. Clearly, $H=\langle\!\langle 
\hat a_0,\hat a_1\rangle\!\rangle$ is invariant under $\hat f$ and also 
$H=\langle\!\langle \hat a_{-1},\hat a_0,\hat a_1\rangle\!\rangle$ is invariant under the 
involution $\hat \tau_0$. Thus $H$ is invariant under $\hat \theta$ and so $H$ contains all the $\hat a_i$'s. It follows that $H$ contains all the $s_{\bar  r, j}$, that is $H=\h$.  

Since, for every $i\in \Z$, $\hat a_i=\hat a_0^{\hat \theta_i}$, to show that $\h$ is an axial algebra of Monster type $(2, \frac{1}{2})$ it is enough to prove that $\hat a_0$ is an axis with respect to the Monster fusion law in Table~\ref{Htable}.
\begin{table}
$$ 
\begin{array}{|c||c|c|c|c|}
\hline
\star & 1 & 0 & 2 & -2\\
\hline
\hline
1 & 1 & & 2 & -2\\
\hline
0 & & 0 & 2 & -2\\
\hline
2 & 2 & 2 & 0,1 & -2\\
\hline
-2 & -2 & -2 & -2 & 0,1,2\\
\hline
\end{array}
$$
\caption{Fusion law for $\h$}\label{Htable}
\end{table} 
Further, for every $i\in \Z_+$, we have
\begin{equation}\label{basis1}
\langle \hat a_0, \hat a_i, \hat a_{-i}, \hat s_{\bar  0, i}\rangle=\langle \hat a_0, \hat u_i, \hat v_{i}, \hat w_i\rangle \mbox{ if }i\not \equiv_3 0
\end{equation}
and 
\begin{equation}\label{basis2}
\langle \hat a_0, \hat a_i, \hat a_{-i}, \hat s_{\bar  0, i}, \hat s_{\bar  1, i}, \hat s_{\bar  2, i}\rangle=\langle \hat a_0, \hat u _i, \hat v_{i}, \hat w_i, \overline u_i, \overline w_i\rangle \mbox{ if } i \equiv_3 0.
\end{equation}
Hence, a basis of $\ad_{\hat a_0}$-eigenvectors for $\h$ is given by  
$$
\hat a_0, \: \hat u_i, \:\hat v_i, \:\hat w_i\:, \:\overline u_{3k}, \: \overline w_{3k},\mbox{ with }\:  i,k\in \Z_+. $$
In particular, since $\hat a_0$ is the unique element of this basis that is a $1$-eigenvector for $\ad_{\hat a_0}$, the algebra is primitive. Finally, set
$$
H_0:=\langle \hat u_i, \overline u_i\mid i\in \Z_+\rangle, \:H_2:=\langle \hat v_i \mid i\in \Z_+\rangle, \:H_{-2}:=\langle \hat w_i, \overline w_i\mid i\in \Z_+\rangle.
$$
Then, for $z\in \{0,2,-2\}$, $H_z$ is the $z$-eigenspace for $\ad_{\hat a_0}$ and 
$\hat \tau_0$ acts as the identity on $\langle \hat a_0\rangle\oplus H_0\oplus H_2$ and as the multiplication by $-1$ on $H_{-2}$. Since $\tau_0$ is an algebra automorphism, we have $H_{z}H_{-2}\subseteq H_{-2}$ for every $z\in \{0,2\}$ and $H_{-2}H_{-2}\subseteq \langle \hat a_0\rangle\oplus H_0\oplus H_2$. By Lemma~\ref{products u v}, we also have $H_0H_0\subseteq H_0$, $H_0H_2\subseteq H_2$, and  $H_2H_2\subseteq H_0$. Hence $\h$ respects (a restricted version of) the Monster fusion law and the result is proved.
 \end{proof}

Note that $\langle \hat s_{\bar  0,3k}-\hat s_{\bar  1,3k}, \hat s_{\bar  0,3k}-\hat s_{\bar  2,3k}, \hat s_{\bar  1,3k}-\hat s_{\bar  2,3k}\mid  k\in \Z_+\rangle$ is an $\hat f$-invariant ideal of $\h$ and the corresponding factor algebra is isomorphic to the Highwater algebra $\mathcal H$. Moreover, the algebra $6A_2$ is isomorphic to the factor of $\h$ over the ideal linearly spanned by the vectors
$$
\hat a_{i}-\hat a_{i-6}, \mbox{ for } i\geq 3, \:\:\hat s_{\bar  0,4}-\hat s_{\bar  0,2},\: \:\hat s_{\bar  0,5}-\hat s_{\bar  0,1}, \:\: \hat s_{\bar  0,j}-\hat s_{\bar  0,j-6}, \mbox{ for } j\geq 6,\:\:\hat x, \:\: \hat x^{\hat f}, \:\: \hat x^{\hat f \hat \tau_0}, $$
where 
$$ \hat x:=\hat s_{\bar  0,3}-\hat s_{\bar  0,1}-\hat s_{\bar  0,2}+\hat a_{-2}+\hat a_{-1}+\hat a_{1}+\hat a_{2}-2(\hat a_0+\hat a_3).$$

\section{Proofs of the main results}
In the next lemma we recall some basic properties of the elements $s_{\bar  r,n}$ that will be used throughout  the proof of Theorem~\ref{thm1} without  further reference.
\begin{lemma}\label{s}
Let $V$ be a primitive $2$-generated symmetric axial algebra of Monster type. For every $n\in \Z_+$ and $i\in \Z$ the following hold
\begin{enumerate}
\item $(s_{\bar  r,n})^{\sigma_i}=s_{\bar  k,n}$, with $r+i\equiv k \:\bmod n$;
\item  the group $\langle \tau_0, f\rangle$ acts transitively on the set $\{s_{\bar  r,n} \mid \bar r\in \Z/n\Z\}$, for each $n\in \Z_+$;
\item $\lambda_{a_0}(s_{\bar  0,n})=\lm_n-\bt-\bt\lm_n$.
\end{enumerate}
\end{lemma}
\begin{proof}
The first assertion is Lemma~4.2 in~\cite{FMS1} and  (2) and (3) follow immediately. For the last one, note that, for every $x\in V$, we have $\lambda_{a_0}(a_0x)=\lm_{a_0}(x)$ (this follows immediately from the linearity of $\lm_{a_0}$, decomposing $x$ into a sum of $\ad_{a_0}$-eigenvectors). Hence, by the definition of $s_{\bar  0,n}$, we get
$$
 \lambda_{a_0}(s_{\bar  0,n})= \lambda_{a_0}(a_0a_n-\bt(a_0+a_n))= \lm_n-\bt-\bt \lm_n.
 $$
\end{proof}
\medskip

\noindent {\it Proof of Theorem~\ref{thm1}}.
Let $V$ be a primitive $2$-generated symmetric axial algebra of Monster type $(2, \frac{1}{2})$ over a field of characteristic $5$ such that $\lm_1=\lm_2=1$. By~\cite[Lemma~4.4]{FMS1}, for $h\in \Z$, the $\ad_{a_0}$-eigenvectors $u_h$ and $v_h$, defined in Section~1, are as follows (remind that $\frac{1}{2}=-2$ in characteristic $5$)
\begin{equation}\label{ui}
u_h=-2 a_0+( a_h+ a_{-h})+2 s_{\bar  0,h},
\end{equation}
\begin{equation}\label{vi}
v_h= a_0+2( a_h+ a_{-h})-2 s_{\bar  0,h}.
\end{equation}
 By the fusion law, for every $h,k\in \Z$, the following identities hold
\begin{equation}\label{id}
a_0(u_hu_k-v_hv_k+\lambda_{a_0}(v_hv_k)a_0)=0
\end{equation}
and
\begin{equation}\label{id1}
a_0(u_hu_k+u_hv_k)=2 u_hv_k. 
\end{equation}
By~\cite[Lemma~4.3]{FMS1}, for every $j\in \Z$, we have 
$a_0s_{\bar  0,j}=-2a_0+(a_{-j}+a_j)-s_{\bar  0,j}$.
Using the action of the group of automorphisms $\langle \tau_0, f\rangle$, we get, for every $j,k\in \Z$, $r\in \Z$,
\begin{equation}
\label{m2}
a_{r+jk}s_{\bar r,j}=-2a_{r+jk}+(a_{r+j(k-1)}+a_{r+j(k+1)})-s_{\bar  0,j}.
\end{equation}
Set $V_0:=\langle a_i, s_{\bar r,n}\mid i\in \Z, n\in \Z_+,  r\in \Z\rangle$ and, for $t\in \Z$,  denote by $[t]_3$ the congruence class $t+3\Z$, $0\ast [t]_3:=0$, $(-1)\ast [1]_3:=-1$, and $(-1)\ast [2]_3:=1$.
\medskip

\noindent {\bf Claim.} {\it For every $i\in \Z_+$, $r\in \Z$, and $t\in \{0,1, 2\}$
\begin{enumerate}
\item[(i)]   $\lm_i=1$ and $\lambda_{a_0}(s_{\bar r,i})=0$; 
\item[(ii)]  if $i\not \equiv_3 0$, $s_{\bar r,i}=s_{\bar  0,i}$;
\item[(iii)] if $i\equiv_3 0$ and  $r\equiv_3 t$,  $s_{\bar r,i}=s_{\bar t,i}$;
\item[(iv)] for every $l\in \Z$, $j\leq i$, $m\in \Z$, the products $a_l s_{\bar m,j}$ belong to $V_0$ and satisfy the formula
\begin{equation}\label{product}
 a_l s_{\bar  m,j}= -2 a_l+( a_{l-j}+ a_{l+j})- s_{\bar  m,j}+(\delta_{[ l]_3 [m]_3} -1)\ast {[ l- m]_3}( s_{\bar  m-\bar 1,j}-s_{\bar  m+\bar 1,j}).
\end{equation}
\end{enumerate}

}
\medskip

\noindent Note that, by the symmetries of $V$, part (iv) of Claim holds if and only if,  for every $r\in \{0, \dots , j-1\}$, the products  $a_0 s_{\bar r,j}$ satisfy the corresponding formula in Equation~(\ref{product}).

We proceed by induction on $i$.
Let $i=1$. By the hypothesis $\lm_1=1$, hence (i) holds by Lemma~\ref{s}.(4) and  (ii) holds trivially.

Let $i=2$. Again by the hypothesis, $\lm_2=1$, hence (i) holds by Lemma~\ref{s}.(4).
 Equation~(1) in~\cite[Lemma~4.8]{FMS1} becomes $-2(s_{\bar  0,2}-s_{\bar 1,2})=0$, whence 
\begin{equation}
\label{m1}
s_{\bar 1,2}=s_{\bar  0,2}
\end{equation}
and  parts (ii)  and (iv) hold.

Assume  $i\geq 3$ and the result true for every $l\leq i$. 
By the fusion law, $u_1u_i$, $u_1v_i$ are $0$- and  $2$-eigenvectors for $\ad_{a_0}$, respectively. Further, since 
$$
(s_{\bar 1, i+1})^{\tau_0}=s_{-\bar{1}, i+1}=s_{\bar \imath, i+1},
$$
we get that $s_{\bar 1,i+1}-s_{\bar \imath,i+1}$ is negated by the map $\tau_0$, in particular $s_{\bar 1,i+1}-s_{\bar \imath,i+1}$ is a $-2$-eigenvector for $\ad_{a_0}$. 
It follows that 
$$\lambda_{a_0}(u_1 u_j)=\lambda_{a_0}(u_1 v_j)=\lambda_{a_0}(s_{\bar 1,i+1}-s_{\bar \imath,i+1})=0.$$ 
 By Equations~(\ref{ui}) and~(\ref{vi}) and linearity of $\lambda_{a_0}$, we get 
$$
0=\lm_{a_0}(u_1u_i+u_1v_i)=2-2\lambda_{a_0}(s_{\bar 1,i+1})-2\lm_{i+1},
$$
and
$$
0=\lm_{a_0}(u_1u_i)=\lm_{a_0}(s_{\bar  0,1}s_{\bar  0,i})-2\lambda_{a_0}(s_{\bar 1,i+1})+2\lm_{i+1}-2,
$$ 
whence
\begin{equation}\label{lambdas}
\lm_{a_0}(s_{\bar 1,i+1})=1-\lm_{i+1}\:\:\mbox{ and }\:\:
\lm_{a_0}(s_{\bar  0,1}s_{\bar  0,i})=\lm_{i+1}-1.
\end{equation}
As above, substituting $u_1$, $u_i$, $v_1$, and $v_i$ in Equation~(\ref{id}), with $h=1$ and $k=i$,  we get
\begin{equation}\label{a0*sum}
a_0(s_{\bar 1,i+1}+s_{\bar \imath,i+1})=-2(1+\lm_{i+1})a_0+2(a_{i+1}+a_{-i-1})-2s_{\bar  0,i+1}.
\end{equation} 
On the other hand, since $s_{\bar 1,i+1}-s_{\bar \imath,i+1}$ is a $-2$-eigenvector for $\ad_{a_0}$, 
\begin{equation}\label{a0*diff}
a_0(s_{\bar 1,i+1}-s_{\bar \imath,i+1})=-2(s_{\bar 1,i+1}-s_{\bar \imath,i+1}).
\end{equation} 
Taking the sum and the difference of both members of Equations~(\ref{a0*sum}) and~(\ref{a0*diff}) we get
\begin{eqnarray}\label{a0s3f}
a_0s_{\bar 1,i+1}&=&-(1+\lm_{i+1} )a_0+(a_{i+1}+a_{-i-1})-s_{\bar  0,i+1}-s_{\bar 1,i+1}+s_{\bar \imath,i+1}
\end{eqnarray}
and 
\begin{eqnarray}\label{a0s3t}
a_0s_{\bar 2,i+1}&=&-(1+\lm_{i+1} )a_0+(a_{i+1}+a_{-i-1})-s_{\bar  0,i+1}+s_{\bar 1,i+1}-s_{\bar \imath ,i+1}.
\end{eqnarray}
Assume first that $i+1\equiv_3 0$.
Substituting the expressions~(\ref{ui}) and~(\ref{vi}) in Equation~(\ref{id1}), with $h=1$ and $k=i$, and using  Equation~(\ref{a0*sum}) we  get
\begin{equation}\label{s1si}
s_{\bar  0,1}s_{\bar  0,i}=2(\lm_{i+1}-1)a_0-2(s_{\bar 1,i+1}+s_{\bar \imath,i+1}+s_{\bar  0,i+1})+2s_{\bar  0,1}-s_{\bar  0,i-1}+2s_{\bar  0,i}.
\end{equation}
Since, $s_{\bar  0,1}$ and, by the inductive hypothesis, $s_{\bar  0,i-1}$ and $s_{\bar  0,i}$  are $\sigma_j$-invariant for every $j\in \Z$,  subtracting to both members of Equation~(\ref{s1si}) their images under $\sigma_{i+1}$ we get
\begin{equation}\label{mx}
0=s_{\bar  0,1}s_{\bar  0,i}-(s_{\bar  0,1}s_{\bar  0,i})^{\sigma_{i+1}}=2(\lm_{i+1}-1)(a_0-a_{i+1}),
\end{equation}
whence either $\lm_{i+1}=1$ or $a_0=a_{i+1}$. But, again, in the latter case,  $\lm_{i+1}=\lm_{a_0}(a_0)=1$. Hence, by Equation~(\ref{lambdas}), giving (i). 
In particular Equation~(\ref{s1si}) becomes 
\begin{equation}\label{s1sibis}
s_{\bar  0,1}s_{\bar  0,i}=-2(s_{\bar 1,i+1}+s_{\bar \imath,i+1}+s_{\bar  0,i+1})+2s_{\bar  0,1}-s_{\bar  0,i-1}+2s_{\bar  0,i}.
\end{equation}
Again, taking the image under $\sigma_i$ of both members of the above equation, we get  
\begin{equation*}
0=s_{\bar  0,1}s_{\bar  0,i}-(s_{\bar  0,1}s_{\bar  0,i})^{\sigma_i}=-2(s_{\bar 1,i+1}-s_{\bar \imath-\bar 1,i+1}),
\end{equation*}
whence $s_{\bar 1,i+1}=s_{\bar \imath-\bar 1,i+1}$.
 Since, by Lemma~\ref{s}.(3), the group $\langle \tau_0, f\rangle$ is transitive on the set $\{s_{\bar r,i+1}\mid 0\leq r\leq i\}$, it follows that (iii) holds, in particular
 \begin{equation}\label{m16}
 s_{\bar \imath,i+1}=s_{\bar 2,i+1}.
 \end{equation} 
 Hence, in order to prove (iv), we just need to check that Equation~(\ref{product}) holds for $l=0$ and $m\in \{0,1,2\}$. The case $m=0$ follows from Equation~(\ref{m2}), cases $m\in \{1,2\}$ follow from Equations~(\ref{a0s3f}),~(\ref{a0s3t}), and (\ref{m16}). 

Assume now $i+1\equiv_3 1$. 
Substituting the expressions~(\ref{ui}) and~(\ref{vi}) in Equation~(\ref{id1}), with $h=2$ and $k=i-1$, since by the inductive hypothesis (iv), for every $l\in \Z$ and $j\leq i+1$, $r\in \Z$, the products $a_l s_{\bar r,j}$ are given by Equation~(\ref{product}),  we  get
$$
s_{\bar  0,2}s_{\bar  0,i-1}=2(\lm_{i+1}-1)a_0-2(s_{\bar 1,i-3}+s_{\bar 2,i-3}+s_{\bar  0,i-3})+2s_{\bar  0,2}-s_{\bar  0,i+1}+2s_{\bar  0,i-1}.
$$
Since, $s_{\bar  0,2}$ and, by the inductive hypothesis, $s_{\bar  0,i-1}$ and $s_{\bar 1,i-3}+s_{\bar 2,i-3}+s_{\bar  0,i-3}$  are $\sigma_j$-invariant for every $j\in \Z$,  as above we obtain
\begin{equation}\label{}
0=s_{\bar  0,2}s_{\bar  0,i-1}-(s_{\bar  0,2}s_{\bar  0,i-1})^{\sigma_{i+1}}=2(\lm_{i+1}-1)(a_0-a_{i+1}),
\end{equation}
whence, as in the previous case,  $\lm_{i+1}=1$ giving (i) by Equation~(\ref{lambdas}).            
 Similarly, 
\begin{equation*}
0=s_{\bar  0,2}s_{\bar  0,i-1}-(s_{\bar  0,2}s_{\bar  0,i-1})^{\sigma_1}=s_{\bar 1,i+1}-s_{\bar  0,i+1},
\end{equation*}
whence $s_{\bar 1,i+1}=s_{\bar  0,i+1}$. Since the group $\langle \tau_0, f\rangle$ is transitive on the set of all pairs
$(s_{\bar r,i+1}, s_{\bar r+\bar 1,i+1})$,
 it follows that $s_{\bar r,i+1}=s_{\bar  0,i+1}$, for every $r\in \Z$. 

Finally, assume $i+1\equiv_3 2$.
Substituting the expressions~(\ref{ui}) and~(\ref{vi}) in Equation~(\ref{id1}), with $h=1$ and $k=i$, using the inductive hypothesis as above (in particular $s_{\bar i-2,i-1}=s_{\bar 2, i-1}$), we  get
$$
s_{\bar  0,1}s_{\bar  0,i}=2(\lm_{i+1}-1)a_0-2(s_{\bar  0,i-1}+s_{\bar 1,i-1}+s_{\bar 2,i-1})+2s_{\bar  0,1}-s_{\bar  0,i+1}+2s_{\bar  0,i}.
$$
Since $s_{\bar  0,1}$ and, by the inductive hypothesis, $s_{\bar  0,i}$ and $s_{\bar  0,i-1}+s_{\bar 1,i-1}+s_{\bar 2,i-1}$  are $\sigma_j$-invariant for every $j\in \Z$,  we have
\begin{equation}\label{}
0=s_{\bar  0,1}s_{\bar  0,i}-(s_{\bar  0,1}s_{\bar  0,i})^{\sigma_{i+1}}=2(\lm_{i+1}-1)(a_0-a_{i+1}),
\end{equation}
whence, as above, it follows $\lm_{i+1}=1$. Hence, by Equation~(\ref{lambdas}), giving (i). Then, 
\begin{equation*}
0=s_{\bar  0,1}s_{\bar  0,i}-(s_{\bar  0,1}s_{\bar  0,i})^{\sigma_1}=s_{\bar 1,i+1}-s_{\bar  0,i+1},
\end{equation*}
whence we conclude as in the previuous case. This finishes the inductive step and the Claim is proved. As a consequence, we get that the subspace $V_0$ is closed with respect to the multiplication by any axes $a_i$. 

We now consider the products $s_{\bar r,i}s_{\bar t,j}$, for $i,j\in \Z_+$, $r, t\in \Z$.  Proceeding as above, by Equation~(\ref{id1}) with $h=i$ and  $k=j$,  we obtain 
\begin{equation}\label{formula1}
s_{\bar  0,i}s_{\bar  0,j}=2(s_{\bar  0,i}+s_{\bar  0,j})-2( s_{\bar  0,{|i-j|}}+ s_{\bar  1, {|i-j|}}+ s_{\bar  2, {|i-j|}}
+ s_{\bar  0,{i+j}}+ s_{\bar  1,{i+j}}+ s_{\bar  2,{i+j}}),
\end{equation}
  if $ \{i,j\}\not \subseteq 3\Z$,
and 
\begin{equation}\label{formula2}
s_{\bar  0,i} s_{\bar  0,j}=2( s_{\bar  0,i}+ s_{\bar  0,j})-(s_{\bar  0,|i-j|}+ s_{\bar  0,i+j}), \:\: \mbox{ if } i\equiv_3j\equiv_3 0.
\end{equation}
If $i\not \equiv_3 0$, then $s_{\bar  0,i} =s_{\bar  1,i}=s_{\bar  2,i}$, thus,   applying $f$ and  $\tau_1$ to Equation~(\ref{formula1}), we get 
$$
s_{\bar  0,i}s_{\bar 1,j}=
2( s_{\bar  0,i}+ s_{\bar  1,j})-2(s_{\bar  0,{|i-j|}}+ s_{\bar  1, {|i-j|}}+ s_{\bar  2, {|i-j|}}
+ s_{\bar  0,{i+j}}+ s_{\bar  1,{i+j}}+ s_{\bar  2,{i+j}}),
$$
and 
$$
s_{\bar  0,i}s_{\bar 2,j}=
2( s_{\bar  0,i}+ s_{\bar  2,j})-2(s_{\bar  0,{|i-j|}}+ s_{\bar  1, {|i-j|}}+ s_{\bar  2, {|i-j|}}
+ s_{\bar  0,{i+j}}+ s_{\bar  1,{i+j}}+ s_{\bar  2,{i+j}}).
$$
 If $i\equiv_3j\equiv_3 0$, in a similar way, from Equation~(\ref{formula2}), we get, for any $r\in \Z$,
$$
s_{\bar  r,i} s_{\bar  r,j}=2( s_{\bar  r,i}+ s_{\bar  0r,j})-(s_{\bar  r,|i-j|}+ s_{\bar  r,i+j}).
$$ 
 The last products needed are $s_{\bar  r,3h}s_{\bar t,3k}$, for $h, k\in \Z_+$, $r,t \in \Z$ with  $r\not \equiv_3 t$. For $k\in \Z_+$, set 
\begin{equation}\label{u3}
\overline{ u}_{3k}:= a_0+2(a_{-3k}+ a_{3k}) -2( s_{\bar  0,3k}+ s_{\bar 1,3k}+ s_{\bar 2,3k}).
\end{equation}
Then, by Equation~(\ref{product}), $\overline{ u}_{3k}$ is a $0$-eigenvector for $\ad_{a_0}$. Hence, 
 by the fusion law, we have 
 $$a_0(\overline{u}_{3k}u_{3h}+\overline{u}_{3k}v_{3h})=2 \overline{u}_{3k}v_{3h}.
 $$  
 Substituting the expressions~(\ref{u3}) and (\ref{vi}) in the above equation, we get 
\begin{equation}\label{s3*sum}
s_{\bar  0,3h}(s_{\bar 1,3k}+s_{\bar 2,3k})=-s_{\bar  0,3h}-s_{\bar  0,3k}-2(s_{\bar  0,3|h-k|}+s_{\bar  0,3(h+k)}), 
\end{equation}
whence, taking the images  under $f$ of both sides, we get 
\begin{equation}\label{s3f*sum}
s_{\bar 1,3h}(s_{\bar  0,3k}+s_{\bar 2,3k})=-s_{\bar 1,3h}-s_{\bar 1,3k}-2(s_{\bar 1,3|h-k|}+s_{\bar 1,3(h+k)}). 
\end{equation}
Taking the difference of the Equations~(\ref{s3*sum}) and~(\ref{s3f*sum}), we obtain
\begin{eqnarray}\label{s3t*diff}
s_{\bar 2,3k}(s_{\bar  0,3h}-s_{\bar 1,3h})&=&-(s_{\bar  0,3h}-s_{\bar 1,3h}+s_{\bar  0,3k}-s_{\bar 1,3k})\\
&&-2(s_{\bar  0,3|h-k|}-s_{\bar 1,3|h-k|}+s_{\bar  0,3(h+k)}-s_{\bar 1,3(h+k)}). \nonumber
\end{eqnarray}
Since in Equation~(\ref{s3t*diff})  we can swap $h$ and $k$, we finally get
\begin{eqnarray*}\label{s3s3}
s_{\bar  0,3h}s_{\bar 1,3k}&=&\frac{1}{2}\left [ s_{\bar  0,3h}(s_{\bar 1,3k}+s_{\bar 2,3k})-(s_{\bar 2,3h}(s_{\bar  0,3k}-s_{\bar 1,3k}))^{\tau_1}\right ]\\
&=&2(s_{\bar  0,3h}+s_{\bar 1,3h}-s_{\bar 2,3h}+s_{\bar  0,3k}+s_{\bar 1,3k}-s_{\bar 3k})\\
&&-(s_{\bar  0,3|h-k|}+s_{\bar 1,3|h-k|}-s_{\bar 2,3|h-k|}+s_{\bar  0,3(h+k)}+s_{\bar 1,3(h+k)}-s_{\bar 3(h+k)}).
\end{eqnarray*}
Using the maps $\tau_0$ and $f$, we derive the formulas for the products $s_{\bar  0,3h}s_{\bar 2,3k}$ and $s_{\bar  1,3h}s_{\bar 2,3k}$. Hence $V_0$ is a subalgebra of $V$, and since $a_0, a_1\in V_0$, we get $V_0=V$. Therefore, the map
$$\begin{array}{cllcl}
\phi&\colon & \h &\to& V\\
 & &\hat a_i &\mapsto &a_i\\
 && \hat s_{\bar  r, j}& \mapsto & s_{\bar r, j}
 \end{array}
$$ 
from the basis $\mathcal B$ of $\h$ and $V$, extends to a surjective linear map $\bar\phi:\h \to V$. $\bar \phi$ is actually an algebra homomorphism, since  $V$ satisfies the multiplication table of the algebra $\h$ and the result follows.
\hfill $\square$
\medskip 

\noindent {\it Proof of Theorem~\ref{thms}}.
Let $V$ be a primitive $2$-generated axial algebra of Monster type $(\al, \bt)$ over a field $\F$ of characteristic $5$,  such that $D\geq 6$. If $\al=4\bt$, condition $D\geq 6$ and the proof of Proposition~4.12 in~\cite{FMS1} yield $(\al,\bt)=(2,\frac{1}{2})$ and $\lm_1=\lm_2=1$. Then, by Theorem~\ref{thm1}, $V$ satisfies (3).  If $\al=2\bt$, the proof of Claim 5.18 in~\cite{Yabe} is still valid in characteristic $5$, proving (2). Similarly, if $\al\neq 4\bt, 2\bt$, the proof of Claim 5.19 in~\cite{Yabe} is also valid in characteristic $5$ and gives (1).  \hfill  $\square$

\end{document}